\documentclass[12pt]{article}
\usepackage{amssymb}
\usepackage{amsthm}
\usepackage{amsmath}
\usepackage{graphicx}
\usepackage{enumerate}
\usepackage{pict2e}
\usepackage{framed}

\DeclareMathOperator{\Max}{Max}
\DeclareMathOperator{\Min}{Min}
\DeclareMathOperator{\C}{{\rm C}}

\newtheorem{theorem}{Theorem}[section]
\newtheorem{definition}[theorem]{Definition}
\newtheorem{lemma}[theorem]{Lemma}
\newtheorem{proposition}[theorem]{Proposition}

\newtheorem{remark}[theorem]{Remark}
\newtheorem{example}[theorem]{Example}
\newtheorem{corollary}[theorem]{Corollary}
\title{Derived operators on skew orthomodular and strong skew orthomodular posets}
\author{Ivan~Chajda and Helmut~L\"anger}
\date{}

\begin{document}

\footnotetext{Support of the research of the first author by the Czech Science Foundation (GA\v CR), project 24-14386L, entitled ``Representation of algebraic semantics for substructural logics'', and by IGA, project P\v rF~2025~008, is gratefully acknowledged.}
	
\maketitle
	
\begin{abstract}
It is well-known that in the logic of quantum mechanics disjunctions and conjunctions can be represented by joins and meets, respectively, in an orthomodular lattice provided their entries commute. This was the reason why J.~Pykacz introduced new derived operations called ``sharp'' and ``flat'' coinciding with joins and meets, respectively, for commuting elements but sharing some appropriate properties with disjunction and conjunction, respectively, in the whole orthomodular lattice in question. The problem is that orthomodular lattices need not formalize the logic of quantum mechanics since joins may not be defined provided their entries do not commute. A corresponding fact holds for meets. Therefore, orthomodular posets are more accepted as an algebraic formalization of such a logic. The aim of the present paper is to extend the concepts of ``sharp'' and ``flat'' operations to operators in skew orthomodular and strong skew orthomodular posets. We generalize the relation of commuting elements as well as the commutator to such posets and we present some important properties of these operators and their mutual relationships. Moreover, we show that if the poset in question is even Boolean then there can be defined a ternary operator sharing the identities of a Pixley term. Finally, under some weak conditions which are automatically satisfied in Boolean algebras we show some kind of adjointness for operators formalizing conjunction and implication, respectively.
\end{abstract}
	
	

\section{Introduction}

It is a characteristic feature of quantum mechanics that there do exist experimentally testable propositions on quantum physical systems that cannot be tested simultaneously. In this respect, the logic of quantum mechanics differs a lot from classical physics. Such propositions in the logic of quantum mechanics are represented in orthomodular lattices, which are believed to be an algebraic representation of families of experimentally testable propositions about quantum systems, i.e.\ by elements that are compatible, see e.g.\ \cite B, \cite{DP}, \cite{PP} and \cite P. Therefore conjunctions and disjunctions introduced by G.~Birkhoff and J.~von~Neumann \cite{Bv} should be represented by meets and joins of pairwise compatible elements, in order to be able to confirm or falsify the corresponding propositions by means of experiments. Hence, conjunctions and disjunctions of two propositions in quantum systems could be represented by other operations in orthomodular lattices that should coincide with meets and joins of pairwise compatible elements. In his paper \cite P Pykacz introduced such operations called ``sharp'' and ``flat'' as certain terms in orthomodular lattices.

However, there is also another problem connected with the logic of quantum mechanics, namely joins and meets of elements of the lattice in question need not be defined provided the entries are neither comparable nor orthogonal. One way how to overcome this problem is to consider orthomodular posets instead of an orthomodular lattices. Of course, then the theory is more complicated and not so many papers are devoted to such posets. Posets of this kind are treated e.g.\ in \cite{CL} and \cite{PP}. Hence, it is a natural question how to transform the sharp and flat operations introduced by J.~Pykacz to posets with complementation, in particular to posets which are skew orthomodular or even orthomodular. Our present paper is devoted to this study. Notice that skew orthomodular posets were introduced by the authors in \cite{CL}.

\section{Basic concepts}

We start by recalling several concepts connected with posets.

In the following we identify singletons with their unique element.

Let $\mathbf P=(P,\le)$ be a poset, $a,b\in P$ and $A,B\subseteq P$. We define
\begin{align*}
A\le B & \text{ if }x\le y\text{ for all }x\in A\text{ and all }y\in B, \\
  L(A) & :=\{x\in P\mid x\le A\}, \\
  U(A) & :=\{x\in P\mid A\le x\}.
\end{align*}
Instead of $L(\{a,b\})$, $L(A\cup\{a\})$, $L(A\cup B)$ and $L\big(U(A)\big)$ we simply write $L(a,b)$, $L(A,a)$, $L(A,B)$ and $LU(A)$, respectively. Similarly we proceed in analogous cases. If $\mathbf P$ is bounded then $L(A),U(A)\ne\emptyset$. We say that $\mathbf P$ satisfies the {\em Ascending Chain Condition {\rm(ACC)}} or the {\em Descending Chain Condition {\rm(DCC)}} if $\mathbf P$ has no infinite ascending  or descending chain, respectively. Moreover, let {\em $\Min A$} and {\em $\Max A$} denote the set of all minimal and maximal elements of $A$, respectively. If $\mathbf P$ satisfies the ACC and the DCC and $A\ne\emptyset$ then $\Min A,\Max A\ne\emptyset$.

By an {\em orthoposet} we mean a bounded poset $(P,\le,{}',0,1)$ with an antitone involution $'$ that is a complementation. For $a,b\in P$ we write $a\perp b$ if $a\le b'$. In this case we call $a$ and $b$ {\em orthogonal} to each other. In orthoposets the conditions (ACC) and (DCC) are equivalent. Recall that an {\em orthomodular poset} is an orthoposet $(P,\le,{}',0,1)$ having the property that $x\vee y$ exists provided $x\perp y$ and satisfying one of the following equivalent conditions:
\begin{align*}
x\le y & \text{ implies }y=x\vee(y\wedge x'), \\
x\le y & \text{ implies }x=y\wedge(x\vee y')
\end{align*}
for all $x,y\in P$. An {\em orthomodular lattice} is an orthomodular poset that is a lattice.

\section{Skew orthomodular posets}

In \cite{CKL} together with M.~Kola\v r\'ik we introduced the following notion:

A {\em skew orthomodular poset} is an orthoposet $(P,\le,{}',0,1)$ satisfying one of the following equivalent conditions:
\begin{align*}
x\le y & \text{ implies }U(y)=U\big(x,L(y,x')\big), \\
x\le y & \text{ implies }L(x)=L\big(y,U(x,y')\big)
\end{align*}
for all $x,y\in P$, i.e.\ the {\em skew orthomodular law}. Skew orthomodular posets were originally introduced by the present authors under the name {\em generalized orthomodular posets} in \cite{CL}.

The following elementary lemma shows that a skew orthomodular poset cannot contain a sublattice of the form depicted in Fig.~1.

\vspace*{-4mm}

\begin{center}
\setlength{\unitlength}{7mm}
\begin{picture}(6,8)
\put(3,1){\circle*{.3}}
\put(1,3){\circle*{.3}}
\put(5,3){\circle*{.3}}
\put(1,5){\circle*{.3}}
\put(5,5){\circle*{.3}}
\put(3,7){\circle*{.3}}
\put(3,1){\line(-1,1)2}
\put(3,1){\line(1,1)2}
\put(3,7){\line(-1,-1)2}
\put(3,7){\line(1,-1)2}
\put(1,3){\line(0,1)2}
\put(5,3){\line(0,1)2}
\put(2.85,.3){$0$}
\put(.35,2.85){$a$}
\put(5.4,2.85){$b'$}
\put(.35,4.85){$b$}
\put(5.4,4.85){$a'$}
\put(2.85,7.4){$1$}
\put(2.2,-.75){{\rm Fig.~1}}
\end{picture}
\end{center}

\vspace*{4mm}

\begin{lemma}
Let $(P,\le,{}',0,1)$ be a skew orthomodular poset and $a,b\in P$ with $a\le b$. Then the following holds:
\begin{enumerate}[{\rm(i)}]
\item If $L(b,a')=0$ then $a=b$,
\item if $U(a,b')=1$ then $a=b$.
\end{enumerate}
\end{lemma}

\begin{proof}
\
\begin{enumerate}[(i)]
\item Using skew orthomodularity we obtain $U(b)=U\big(a,\L(b,a')\big)=U(a,0)=U(a)$.
\item follows from (i) by duality.
\end{enumerate}
\end{proof}

Of course, not every orthoposet is orthomodular. However, we can easily show the connection between orthomodular posets and skew orthomodular posets.

\begin{lemma}
Let $\mathbf P=(P,\le,{}',0,1)$ be an orthoposet such that $x\vee y$ is defined provided $x\perp y$. Then $\mathbf P$ is an orthomodular poset if and only if it is a skew orthomodular poset.
\end{lemma}

\begin{proof}
Let $a,b\in P$ with $a\le b$. Then any of the following statements implies the next one: $a\le b$, $a\perp b'$, $a\vee b'$ is defined, $a'\wedge b$ is defined, $a'\wedge b\perp a$, $a\vee(a'\wedge b)$ is defined. Now
\[
U\big(a,L(b,a')\big)=U\big(a,L(b\wedge a')\big)=U(a,b\wedge a')=U\big(a\vee(b\wedge a')\big).
\]
Hence the following are equivalent: $b=a\vee(b\wedge a')$, $U(b)=U\big(a\vee(b\wedge a')\big)$, $U(b)=U\big(a,L(b,a')\big)$.	
\end{proof}

Strengthening the definition of a skew orthomodular poset we define

\begin{definition}
A {\em strong skew orthomodular poset} is an orthoposet $(P,\le,{}',0,1)$ satisfying one of the following equivalent conditions:
\begin{align*}
A\le y & \text{ implies }U(y)=U\big(A,L(y,A')\big), \\
x\le A & \text{ implies }L(x)=L\big(A,U(x,A')\big)
\end{align*}
for all $A\subseteq P$ and all $x,y\in P$ where $A':=\{x'\mid x\in A\}$.
\end{definition}

Similarly as in the case of skew orthomodular posets, also strong skew orthomodular posets cannot contain a sublattice of the form depicted in Fig.~1.

\begin{lemma}
Let $(P,\le,{}',0,1)$ be a strong skew orthomodular poset, $A\subseteq P$ and $a\in P$. Then the following holds:
\begin{enumerate}[{\rm(i)}]
\item If $A\le a$ and $L(a,A')=0$ then $\bigvee A=a$,
\item if $a\le A$ and $U(a,A')=1$ then $\bigwedge A=a$.
\end{enumerate}
\end{lemma}

\begin{proof}
\
\begin{enumerate}[(i)]
\item Using strong skew orthomodularity we compute $U(a)=U\big(A,\L(a,A')\big)=U(A,0)=U(A)$.
\item follows from (i) by duality.
\end{enumerate}
\end{proof}

The next results show the role of antichains of maximal respectively minimal elements in upper respectively lower cones and in posets with an antitone involution.

\begin{lemma}\label{lem2}
Let $\mathbf P=(P,\le)$ be a poset and $A\subseteq P$. Then the following holds:
\begin{enumerate}[{\rm(i)}]
\item If $\mathbf P$ satisfies the {\rm ACC} then $U(A)=U(\Max A)$,		
\item if $\mathbf P$ satisfies the {\rm DCC} then $L(A)=L(\Min A)$.
\end{enumerate}
\end{lemma}

\begin{proof}
\
\begin{enumerate}[(i)]
\item Assume $\mathbf P$ to satisfy the ACC. Since $\Max A\subseteq A$ we have $U(A)\subseteq U(\Max A)$. Conversely, let $a\in U(\Max A)$. If $b\in A$ then according to the ACC there exists some $c\in\Max A$ with $b\le c$ and hence $b\le c\le a$ implying $b\le a$. This shows $a\in U(A)$ and hence $U(\Max A)\subseteq U(A)$.
\item follows by duality.
\end{enumerate}
\end{proof}

The proof of the following lemma is straightforward.

\begin{lemma}\label{lem4}
If $(P,\le,{}')$ is a poset with an antitone involution and $A\subseteq P$ then $(\Max A)'=\Min A'$ and $(\Min A)'=\Max A'$.	
\end{lemma}

\begin{corollary}\label{cor1}
Let $\mathbf P=(P,\le,{}',0,1)$ be an orthoposet satisfying the {\rm ACC}. Then $\mathbf P$ is a strong skew orthomodular poset if and only if $U(a)=U\big(A,L(a,A')\big)$ for every antichain $A$ of $(P,\le)$ and every $a\in U(A)$.	
\end{corollary}

\begin{proof}
According to Lemma~\ref{lem2} and \ref{lem4} we have
\begin{align*}
U\big(A,L(a,A')\big) & =U(A)\cap U\big(L(a)\cap L(A')\big)=U(\Max A)\cap U\big(L(a)\cap L(\Min A')\big)= \\
                     & =U(\Max A)\cap U\Big(L(a)\cap L\big((\Max A)'\big)\Big)=U\Big(\Max A,L\big(a,(\Max A)'\big)\Big),
\end{align*}
and $\Max A$ is an antichain of $(P,\le)$. Hence using Lemma~\ref{lem2} the following are equivalent:
\begin{align*}
U(a) & =U\big(A,L(a,A')\big)\text{ for all }A\subseteq P\text{ and all }a\in U(A), \\
U(a) & =U\Big(\Max A,L\big(a,(\Max A)'\big)\Big)\text{ for all }A\subseteq P\text{ and all }a\in U(\Max A).
\end{align*}
\end{proof}

Next we present some examples of strong skew orthomodular posets.

\begin{example}\label{ex1}
Let $A$ and $A'$ be disjoint sets of the same cardinality and $\,'$ a bijection from $A$ to $A'$, assume $0,1\notin A\cup A'$ and put $P:=A\cup A'\cup\{0,1\}$. Extend $\,'$ to $P$ by defining $0':=1$, $1':=0$ and $(x')':=x$ for all $x\in A$. Define a binary relation $\le$ on $P$ by $x\le y$ if $x=y$ or $x=0$ or $y=1$ or {\rm(}$x\in A$ and $y\in A'\setminus\{x'\}${\rm)}. Then $\mathbf P:=(P,\le,{}',0,1)$ is a strong skew orthomodular poset. This can be seen as follows: Obviously, $\mathbf P$ is an orthoposet satisfying the {\rm ACC}. We apply Corollary~\ref{cor1}. The poset $(P,\le)$ has the following antichains: all subsets of $P$ of cardinality $\le1$, all subsets of $A$, all subsets of $A'$ and all sets of the form $\{x,x'\}$ with $x\in A$. Let now $B$ be a antichain of $(P,\le)$ and $a\in U(B)$. \\
If $B=\emptyset$ then $U\big(B,L(a,B')\big)=U\big(\emptyset,L(a,\emptyset)\big)=UL(a)=U(a)$, \\
if $B=\{0\}$ then $U\big(B,L(a,B')\big)=U\big(0,L(a,1)\big)=UL(a)=U(a)$, \\
if $B=\{b\}$ with $b\in A$ and $a=b$ then $U\big(B,L(a,B')\big)=U\big(b,L(b,b')\big)=U(b)$, \\
if $B=\{b\}$ with $b\in A$ and $a=c'$ with $c\in A\setminus\{b\}$ then $U\big(B,L(a,B')\big)=U\big(b,L(c',b')\big)=U(b,A\setminus\{b,c\})=U(A\setminus\{c\})=U(c')$, \\
if $B=\{b\}$ with $b\in A$ and $a=1$ then $U\big(B,L(a,B')\big)=U\big(b,L(1,b')\big)=U(b,b')=U(1)$, \\
if $B=\{b'\}$ with $b\in A$ and $a=b'$ then $U\big(B,L(a,B')\big)=U\big(b',L(b',b)\big)=U(b')$, \\
if $B=\{b'\}$ with $b\in A$ and $a=1$ then $U\big(B,L(a,B')\big)=U\big(b',L(1,b)\big)=U(b',b)=U(1)$, \\
if $B=\{1\}$ and $a=1$ then $U\big(B,L(a,B')\big)=U\big(1,L(1,0)\big)=U(1)$, \\
if $B=\{b,b'\}$ with $b\in A$ and $a=1$ then $U\big(B,L(a,B')\big)==U\big(b,b',L(1,b',b)\big)=U(1)$, \\
if $B\subseteq A$, $|B|\ge2$ and $a=b'$ with $b\in A\setminus B$ then $U\big(B,L(a,B')\big)==U\big(B,A\setminus(B\cup\{b\})\big)=U(A\setminus\{b\})=U(b')$, \\
if $B\subseteq A$, $|B|\ge2$ and $a=1$ then $U\big(B,L(a,B')\big)=U\big(B,L(1,B')\big)=U(B,A\setminus B)=U(A)=U(1)$, \\
if $B\subseteq A'$, $|B|\ge2$ and $a=1$ then $U\big(B,L(a,B')\big)=U(1)$. \\
If $|A|\in\{0,1,3\}$ then $\mathbf P$ is a Boolean algebra, if $|A|=2$, say $A=\{a,b\}$, then $\mathbf P$ is a lattice, but not orthomodular since $a\le b'$, but $a\vee(b'\wedge a')=a\vee0=a\ne b'$. If $|A|\ge4$ then $\mathbf P$ is not an orthomodular poset since $c'$ and $d'$ are distinct minimal upper bounds of the mutually orthogonal elements $a$ and $b$ provided $a,b,c,d$ are four pairwise distinct elements of $A$.
\end{example}

The following strong skew orthomodular poset that is neither orthomodular nor a lattice is a particular case of the poset from Example~\ref{ex1} and it is visualized in Fig.~2.

\vspace*{-4mm}

\begin{center}
\setlength{\unitlength}{7mm}
\begin{picture}(8,8)
\put(4,1){\circle*{.3}}
\put(1,3){\circle*{.3}}
\put(3,3){\circle*{.3}}
\put(5,3){\circle*{.3}}
\put(7,3){\circle*{.3}}
\put(1,5){\circle*{.3}}
\put(3,5){\circle*{.3}}
\put(5,5){\circle*{.3}}
\put(7,5){\circle*{.3}}
\put(4,7){\circle*{.3}}
\put(4,1){\line(-3,2)3}
\put(4,1){\line(-1,2)1}
\put(4,1){\line(1,2)1}
\put(4,1){\line(3,2)3}
\put(4,7){\line(-3,-2)3}
\put(4,7){\line(-1,-2)1}
\put(4,7){\line(1,-2)1}
\put(4,7){\line(3,-2)3}
\put(1,3){\line(0,1)2}
\put(1,3){\line(1,1)2}
\put(1,3){\line(2,1)4}
\put(3,3){\line(-1,1)2}
\put(3,3){\line(0,1)2}
\put(3,3){\line(2,1)4}
\put(5,3){\line(-2,1)4}
\put(5,3){\line(0,1)2}
\put(5,3){\line(1,1)2}
\put(7,3){\line(-2,1)4}
\put(7,3){\line(-1,1)2}
\put(7,3){\line(0,1)2}
\put(3.85,.3){$0$}
\put(.35,2.85){$a$}
\put(2.35,2.85){$b$}
\put(5.4,2.85){$c$}
\put(7.4,2.85){$d$}
\put(.35,4.85){$d'$}
\put(2.35,4.85){$c'$}
\put(5.4,4.85){$b'$}
\put(7.4,4.85){$a'$}
\put(3.85,7.4){$1$}
\put(3.2,-.75){{\rm Fig.~2}}
\end{picture}
\end{center}

\vspace*{4mm}

\section{Sharp and flat operators in posets}

In the following let $(P,\le,{}')$ be a poset with an antitone involution. We introduce the following operators on $P$ respectively binary relations on $P$:
\begin{align*}
        S(x,y) & :=U\big(L(x,y),L(x,y'),L(x',y)\big), \\
        F(x,y) & :=L\big(U(x,y),U(x,y'),U(x',y)\big), \\
        c(x,y) & :=U\big(L(x,y),L(x,y'),L(x',y),L(x',y')\big), \\
x\mathrel{\C}y & \text{ if }U(x)=U\big(L(x,y),L(x,y')\big)
\end{align*}
for all $x,y\in P$. Using the analogous terminology for lattices introduced by J.~Pykacz in \cite P, we call $S$ ``sharp operator'' and $F$ ``flat operator''. The operator $c$ is called ``commutator''. Analogously as in \cite B and \cite K, the relation $C$ is called ``compatibility relation'' and $a\in P$ is said to be compatible with $b\in P$ provided $a\mathrel{\C}b$.

Recall that a poset $(P,\le)$ is called {\em distributive} if it satisfies one of the following equivalent identities:
\begin{align*}
 L\big(U(x,y),z\big) & \approx LU\big(L(x,z),L(y,z)\big), \\
UL\big(U(x,y),z\big) & \approx U\big(L(x,z),L(y,z)\big), \\
 U\big(L(x,y),z\big) & \approx UL\big(U(x,z),U(y,z)\big), \\
LU\big(L(x,y),z\big) & \approx L\big(U(x,z),U(y,z)\big),
\end{align*}
and that a {\em Boolean poset} is a distributive orthoposet. The following lemma was proved in \cite{CKL}.

\begin{lemma}
Let $\mathbf P=(P,\le,{}',0,1)$ be a Boolean poset. Then the following holds:
\begin{enumerate}[{\rm(i)}]
\item $\mathbf P$ is a skew orthomodular poset,
\item $x\mathrel{\C}y$ for all $x,y\in P$.
\end{enumerate}
\end{lemma}

\begin{proof}
\
\begin{enumerate}[(i)]
\item If $a,b\in P$ and $a\le b$ then
\begin{align*}
U(b) & =UL(b)=UL(b,1)=UL\big(U(b),U(1)\big)=UL\big(U(a,b),U(a,a')\big)= \\
     & =U\big(a,L(b,a')\big).
\end{align*}
\item was proved in \cite{CKL}.
\end{enumerate}
\end{proof}

A Boolean poset which is not a lattice is depicted in Fig.~3.

\vspace*{-4mm}

\begin{center}
\setlength{\unitlength}{7mm}
\begin{picture}(8,10)
\put(4,1){\circle*{.3}}
\put(1,3){\circle*{.3}}
\put(3,3){\circle*{.3}}
\put(5,3){\circle*{.3}}
\put(7,3){\circle*{.3}}
\put(1,7){\circle*{.3}}
\put(3,7){\circle*{.3}}
\put(5,7){\circle*{.3}}
\put(7,7){\circle*{.3}}
\put(4,9){\circle*{.3}}
\put(1,5){\circle*{.3}}
\put(7,5){\circle*{.3}}
\put(4,1){\line(-3,2)3}
\put(4,1){\line(-1,2)1}
\put(4,1){\line(1,2)1}
\put(4,1){\line(3,2)3}
\put(4,9){\line(-3,-2)3}
\put(4,9){\line(-1,-2)1}
\put(4,9){\line(1,-2)1}
\put(4,9){\line(3,-2)3}
\put(1,3){\line(0,1)4}
\put(1,3){\line(1,1)4}
\put(3,3){\line(-1,1)2}
\put(3,3){\line(1,1)4}
\put(5,3){\line(-1,1)4}
\put(5,3){\line(1,1)2}
\put(7,3){\line(-1,1)4}
\put(7,3){\line(0,1)4}
\put(1,5){\line(1,1)2}
\put(7,5){\line(-1,1)2}
\put(3.85,.3){$0$}
\put(.35,2.85){$a$}
\put(2.35,2.85){$b$}
\put(5.4,2.85){$c$}
\put(7.4,2.85){$d$}
\put(.35,4.85){$e$}
\put(7.4,4.85){$e'$}
\put(.35,6.85){$d'$}
\put(2.35,6.85){$c'$}
\put(5.4,6.85){$b'$}
\put(7.4,6.85){$a'$}
\put(3.85,9.4){$1$}
\put(3.2,-.75){{\rm Fig.~3}}
\end{picture}
\end{center}

\vspace*{4mm}

At first, we prove some basic properties of the compatibility relation, sharp and flat operators and the operator $c$ in an arbitrary orthoposet. In \cite P J.~Pykacz showed that if $a\mathrel{\C}b$ in an orthomodular lattice then $s(a,b)=a\vee b$ and $f(a,b)=a\wedge b$ for his operations $s$ and $f$. In the next lemma we show an analogous result for orthoposets, namely that if $a\mathrel{\C}b$ and $b\mathrel{\C}a$ then $S(a,b)=U(a,b)$, and if $a'\mathrel{\C}b'$ and $b'\mathrel{\C}a'$ then $F(a,b)=L(a,b)$.

\begin{proposition}\label{prop1}
Let $(P,\le,{}',0,1)$ be an orthoposet and $a,b\in P$. Then the following holds:
\begin{enumerate}[{\rm(i)}]
\item $c(a,b)=c(a,b')=c(a',b)=c(a',b')=c(b,a)=c(b,a')=c(b',a)=c(b',a')$,
\item $a\mathrel{\C}b$ if and only if $a\mathrel{\C}b'$,
\item $S(a,b)=S(b,a)$ and $F(a,b)=F(b,a)$,
\item $F(a,b)=\big(S(a',b')\big)'$ and $S(a,b)=\big(F(a',b')\big)'$,
\item $S(0,a)=S(a,0)=S(a,a)=U(a)$ and $S(a,a')=S(a,1)=S(1,a)=1$,
\item $F(0,a)=F(a,0)=F(a,a')=0$ and $F(a,a)=F(a,1)=F(1,a)=L(a)$,
\item $U(a,b)\subseteq S(a,b)$ and $L(a,b)\subseteq F(a,b)$,
\item $c(a,b)=S(a,b)\cap UL(a',b')$ and $\big(c(a,b)\big)'=F(a,b)\cap LU(a',b')$,
\item if $a\le b$ then $a\mathrel{\C}b$, $b'\mathrel{\C}a'$, $c(a,b)=S(a,b)\cap U(b')$ and $\big(c(a,b)\big)'=F(a,b)\cap L(a')$,
\item if $a\perp b$ then $S(a,b)=U(a,b)$, and if $a'\perp b'$ then $F(a,b)=L(a,b)$,
\item if $a\mathrel{\C}b$ and $b\mathrel{\C}a$ then $S(a,b)=U(a,b)$,
\item if $a'\mathrel{\C}b'$ and $b'\mathrel{\C}a'$ then $F(a,b)=L(a,b)$,
\item if $a\mathrel{\C}b$ and $a'\mathrel{\C}b'$ then $c(a,b)=1$.
\end{enumerate}
\end{proposition}

\begin{proof}
\
\begin{enumerate}
\item[(i)] -- (viii) are clear.
\item[(ix)] If $a\le b$ then $U(a)=UL(a)=U\big(L(a),L(a,b')\big)=U\big(L(a,b),L(a,b')\big)$, i.e.\ $a\mathrel{\C}b$, and $b'\le a'$ and hence $b'\mathrel{\C}a'$. The rest follows from (viii).
\item[(x)] If $a\perp b$ then $S(a,b)=U\big(L(a,b),L(a),L(b)\big)=U(a,b)$, and if $a'\perp b'$ then $F(a,b)=L\big(U(a,b),U(a),U(b)\big)=L(a,b)$.
\item[(xi)] If $a\mathrel{\C}b$ and $b\mathrel{\C}a$ then
\begin{align*}
S(a,b) & =U\big(L(a,b),L(a,b'),L(a',b)\big)=U\big(L(a,b),L(a,b'),L(a,b),L(a',b)\big)= \\
       & =U\big(L(a,b),L(a,b')\big)\cap U\big(L(a,b),L(a',b)\big)=U(a)\cap U(b)=U(a,b).
\end{align*}
\item[(xii)] follows from (xi) by duality.
\item[(xiii)] If $a\mathrel{\C}b$ and $a'\mathrel{\C}b$ then
\begin{align*} c(a,b) & =U\big(L(a,b),L(a,b'),L(a',b),L(a',b')\big)= \\
                   	  & =U\big(L(a,b),L(a,b')\big)\cap U\big(L(a',b),L(a',b')\big)=U(a)\cap U(a')=U(a,a')=1.
\end{align*}
\end{enumerate}
\end{proof}

If the orthoposet in question is even skew orthomodular, we can prove the following proposition.

\begin{proposition}\label{prop2}
Let $(P,\le,{}',0,1)$ be a skew orthomodular poset and $a,b\in P$ with $a\le b$. Then the following holds:
\begin{enumerate}[{\rm(i)}]
\item $S(a,b)=U(b)$ and $F(a,b)=L(a)$,
\item $c(a,b)=1$.
\end{enumerate}
\end{proposition}

\begin{proof}
\
\begin{enumerate}[(i)]
\item We have $a\mathrel{\C}b$ by (ix) of Proposition~\ref{prop1} and $b\mathrel{\C}a$ by (i) and hence $S(a,b)=U(a,b)=U(b)$ by (xi) of Proposition~\ref{prop1}. Moreover, we have $b'\le a'$ and hence $a'\mathrel{\C}b'$ by (i) and $b'\mathrel{\C}a'$ by (ix) of Proposition~\ref{prop1} which implies $F(a,b)=L(a,b)=L(a)$ by (xii) of Proposition~\ref{prop1}.
\item According to (ix) of Proposition~\ref{prop1} and (ii) we have $c(a,b)=S(a,b)\cap U(b')=U(b)\cap U(b')=1$.
\end{enumerate}
\end{proof}

The next result shows some connection between the compatibility relation and the commutator in a strong skew orthomodular poset. Moreover, we prove some results concerning the symmetry of the compatibility relation.

\begin{theorem}\label{th1}
Let $\mathbf P=(P,\le,{}',0,1)$ be a strong skew orthomodular poset and $a,b\in P$. Then the following are equivalent:
\begin{enumerate}[{\rm(i)}]
\item $a\mathrel{\C}b$,
\item $b\mathrel{\C}a$,
\item $c(a,b)=1$.
\end{enumerate}
\end{theorem}

\begin{proof}
$\text{}$ \\
(i) $\Rightarrow$ (ii): \\
Since $a\mathrel{\C}b$ we have
\begin{align*}
U(a,b') & =U(a)\cap U(b')=U\big(L(a,b),L(a,b')\big)\cap U(b')=U\big(L(a,b),L(a,b'),b'\big)= \\
        & =U\big(L(a,b),b'\big)
\end{align*}
and hence $L\big(U(a',b'),b\big)=L(a',b)$. Because $L(a,b)\le b$ and $\mathbf P$ is a strong skew orthomodular poset we conclude
\[
U(b)=U\Big(L(a,b),L\big(b,U(a',b')\big)\Big)=U\big(L(a,b),L(a',b)\big)
\]
and hence $b\mathrel{\C}a$. \\
(ii) $\Rightarrow$ (iii): \\
Using the fact that (i) implies (ii) and using (ii) of Proposition~\ref{prop1} we conclude $a\mathrel{\C}b'$, $b'\mathrel{\C}a$, $b'\mathrel{\C}a'$ and $a'\mathrel{\C}b'$. Now $a\mathrel{\C}b$ and $a'\mathrel{\C}b'$ together imply $c(a,b)=1$ according to (xiii) of Proposition~\ref{prop1}. \\
(iii) $\Rightarrow$ (i): \\
Since $L(a,b)\cup L(a,b')\le a$ we have
\[
U(a)=U\Big(L(a,b)\cup L(a,b'),L\big(a,U(a',b')\cup U(a',b)\big)\Big).
\]
Now
\[
L\big(a,U(a',b')\cup U(a',b)\big)\subseteq L\big(U(a,b),U(a,b'),U(a',b),U(a',b')\big)=\big(c(a,b)\big)'=1'=0
\]
and hence $L\big(a,U(a',b')\cup U(a',b)\big)=0$ whence
\[
U(a)=U\big(L(a,b)\cup L(a,b'),0\big)=U\big(L(a,b),L(a,b')\big),
\]
i.e.\ $a\mathrel{\C}b$.
\end{proof}

Conversely, we can also show that symmetry of the compatibility relation forces an orthoposet to be skew orthomodular.

\begin{proposition}
Let $\mathbf P=(P,\le,{}',0,1)$ be an orthoposet with symmetric relation $\C$. Then $\mathbf P$ is a skew orthomodular poset.	
\end{proposition}

\begin{proof}
If $a,b\in P$ and $a\le b$ then $a\mathrel{\C}b$ according to (ix) of Proposition~\ref{prop1}	whence $b\mathrel{\C}a$ by symmetry of $\C$ and hence
\[
U(b)=U\big(L(b,a),L(b,a')\big)=U\big(L(a),L(b,a')\big)=U\big(a,L(b,a')\big).
\]
\end{proof}

We can show that a Boolean poset satisfies one more interesting property. Namely, there exists a ternary operator satisfying formally the same property as the Pixley term in Boolean algebras despite the fact that the results of this term for entries from $P$ need not be elements of $P$ in general but only subsets of $P$.

\begin{proposition}\label{prop3}
Let $\mathbf P=(P,\le,{}',0,1)$ be an orthoposet and define
\[
T(x,y,z):=\Min U\big(L(x,z),L(x,y',z'),L(x',y',z)\big)
\]
for all $x,y,z\in P$. Then the following are equivalent:
\begin{enumerate}[{\rm(i)}]
\item $T$ satisfies the identities $T(x,x,z)\approx z$, $T(x,z,z)\approx x$ and $T(x,y,x)\approx x$,
\item $\mathbf P$ satisfies the identity $\Min UL\big(U(x,x'),y\big)\approx U\big(L(x,y),L(x',y)\big)$.
\end{enumerate}
\end{proposition}

\begin{proof}
We have
\begin{align*}
T(x,x,z) & \approx\Min U\big(L(x,z),L(x,x',z'),L(x',x',z)\big)\approx\Min U\big(L(x,z),0,L(x',z)\big)\approx \\
         & \approx\Min U\big(L(x,z),L(x',z)\big), \\
       z & \approx\Min U(z)\approx\Min UL(z)\approx\Min UL(1,z)\approx\Min UL\big(U(x,x'),z\big), \\
T(x,z,z) & \approx T(z,z,x), \\
T(x,y,x) & \approx\Min U\big(L(x,x),L(x,y',x'),L(x',y',x)\big)\approx\Min U\big(L(x),0,0\big)\approx\Min U(x)\approx x.
\end{align*}
\end{proof}

\begin{corollary}
If $\mathbf P=(P,\le,{}',0,1)$ is a Boolean poset and
\[
T(x,y,z):=\Min U\big(L(x,z),L(x,y',z'),L(x',y',z)\big)
\]
for all $x,y,z\in P$ then $T$ satisfies the identities $T(x,x,z)\approx z$, $T(x,z,z)\approx x$ and $T(x,y,x)\approx x$,
\end{corollary}

\begin{proof}
This follows from Proposition~\ref{prop3}.
\end{proof}

\section{Adjointness in orthoposets}

In the last section we show that in special orthoposets the logical connective conjunction can be defined as an operator in such a way that there exists its adjoint operator formalizing implication. This is important for the logic based on orthoposets since then the derivation rule Modus Ponens holds.

Let $(P,\le,{}',0,1)$ be an orthoposet satisfying the ACC. We define two binary operators on $P$, called {\em sharp conjunction} and {\em sharp implication}, respectively, as follows:
\begin{align*}
x\odot y & :=\Max L\big(y,U(x,y')\big), \\
  x\to y & :=\Min U\big(x',L(x,y)\big)
\end{align*}
for all $x,y\in P$. Moreover, we define the following conditions:
\begin{enumerate}[(1)]
\item $L\Big(U\big(L(x,y),y'\big),y\Big)\subseteq L(x)$ for all $x,y\in P$,
\item $U\Big(L\big(U(x,y),y'\big),y\Big)\subseteq U(x)$ for all $x,y\in P$.
\end{enumerate}

Now we show that for orthoposets satisfying conditions (1) and (2) the introduced operators really form an adjoint pair.

\begin{theorem}\label{th2}
Let $(P,\le,{}',0,1)$ be an orthoposet satisfying the {\rm ACC} as well as conditions {\rm(1)} and {\rm(2)} and let $a,b,c\in P$. Then
\[
a\odot b\le c\text{ if and only if }a\le b\to c.
\]
\end{theorem}

\begin{proof}
The following are equivalent:
\begin{align*}
                 a\odot b & \le c, \\
\Max L\big(b,U(a,b')\big) & \le c, \\
     L\big(b,U(a,b')\big) & \le c, \\
     L\big(b,U(a,b')\big) & \subseteq L(c).
\end{align*}
Moreover, the following are equivalent:
\begin{align*}
                   a & \le b\to c, \\
                   a & \le\Min U\big(b',L(b,c)\big), \\	
                   a & \le U\big(b',L(b,c)\big), \\	
U\big(b',L(b,c)\big) & \subseteq U(a).
\end{align*}
If $a\odot b\le c$ then according to (2)
\begin{align*}
U\big(b',L(b,c)\big) & =U\big(b',L(b)\cap L(c)\big)\subseteq U\Big(b',L(b)\cap L\big(b,U(a,b')\big)\Big)= \\
                     & =U\Big(b',L\big(b,U(a,b')\big)\Big)=U\Big(L\big(U(a,b'),b\big),b'\Big)\subseteq U(a)
\end{align*}
and hence $a\le b\to c$. If, conversely, $a\le b\to c$ then according to (1)
\begin{align*}
L\big(b,U(a,b')\big) & =L\big(b,U(a)\cap U(b')\big)\subseteq L\Big(b,U\big(b',L(b,c)\big)\cap U(b')\Big)= \\
                     & =L\Big(b,U\big(b',L(b,c)\big)\Big)=L\Big(U\big(L(c,b),b'\big),b\Big)\subseteq L(c)
\end{align*}
and hence $a\odot b\le c$.
\end{proof}

\begin{corollary}
Let $\mathbf P=(P,\le,{}',0,1)$ be a Boolean poset. Then $\mathbf P$ satisfies conditions {\rm(1)} and {\rm(2)} and hence the operators $\odot$ and $\to$ form an adjoint pair.
\end{corollary}

\begin{proof}
For $a,b\in P$ we have
\begin{align*}
L\Big(U\big(L(a,b),b'\big),b\Big) & =L\Big(UL\big(U(a,b'),U(b,b')\big),b\Big)=L\Big(UL\big(U(a,b'),1\big),b\Big)= \\
& =L\big(ULU(a,b'),b\big)=L\big(U(a,b'),b\big)=LU\big(L(a,b),L(b',b)\big)= \\
& =LU\big(L(a,b),0\big)=LUL(a,b)=L(a,b)\subseteq L(a), \\	
U\Big(L\big(U(a,b),b'\big),b\Big) & =U\Big(LU\big(L(a,b'),L(b,b')\big),b\Big)=U\Big(LU\big(L(a,b'),0\big),b\Big)= \\
& =U\big(LUL(a,b'),b\big)=U\big(L(a,b'),b\big)=UL\big(U(a,b),U(b',b)\big)= \\
& =UL\big(U(a,b),1\big)=ULU(a,b)=U(a,b)\subseteq U(a).
\end{align*}
\end{proof}

\begin{remark}
Let $(P,\le,{}',0,1)$ be an orthoposet satisfying the {\rm ACC} as well as conditions {\rm(1)} and {\rm(2)} and consider the propositional logic based on this poset. The assertion of Theorem~\ref{th2} yields the following simple derivation:
\[
x\to y=x\to y\text{ implies }(x\to y)\odot x\le y.
\]
But this is just the derivation rule Modus Ponens saying that the logical value of proposition $y$ cannot be smaller than the logical value of the conjunction of $x$ and $x\to y$.
\end{remark}








Authors' addresses:

Ivan Chajda \\
Palack\'y University Olomouc \\
Faculty of Science \\
Department of Algebra and Geometry \\
17.\ listopadu 12 \\
771 46 Olomouc \\
Czech Republic \\
ivan.chajda@upol.cz

Helmut L\"anger \\
TU Wien \\
Faculty of Mathematics and Geoinformation \\
Institute of Discrete Mathematics and Geometry \\
Wiedner Hauptstra\ss e 8-10 \\
1040 Vienna \\
Austria, and \\
Palack\'y University Olomouc \\
Faculty of Science \\
Department of Algebra and Geometry \\
17.\ listopadu 12 \\
771 46 Olomouc \\
Czech Republic \\
helmut.laenger@tuwien.ac.at
\end{document}